\documentclass[12pt]{article}%
\usepackage{amsmath}
\usepackage{amsfonts}
\usepackage{amssymb}
\usepackage{graphicx}%
\providecommand{\U}[1]{\protect\rule{.1in}{.1in}}
\newtheorem{theorem}{Theorem}

\newtheorem{corollary}[theorem]{Corollary}

\newtheorem{definition}[theorem]{Definition}
\newtheorem{example}[theorem]{Example}

\newtheorem{lemma}[theorem]{Lemma}

\newtheorem{proposition}[theorem]{Proposition}
\newtheorem{remark}[theorem]{Remark}

\newenvironment{proof}[1][Proof]{\noindent\textbf{#1.} }{\ \rule{0.5em}{0.5em}}
\begin{document}

\title{Generating functions in Riesz spaces}
\author{Youssef Azouzi\thanks{The authors are members of the GOSAEF research group}
and Youssef Nasri \thanks{This document is the results of the research project
funded by the National Science Foundation}\\{\small Research Laboratory of Algebra, Topology, Arithmetic, and Order}\\{\small Department of Mathematics}\\COSAEF {\small Faculty of Mathematical, Physical and Natural Sciences of
Tunis}\\{\small Tunis-El Manar University, 2092-El Manar, Tunisia}}
\maketitle

\begin{abstract}
We introduce and study the concept of generating function for natural elements
in a Dedekind complete Riesz space equipped with a conditional expectatnion
operator. This allows to study discrete processes in free-measure setting. In
particular we improve a result obtained by Kuo, Vardy and Watson concerning
Poisson approximation.

\end{abstract}


\section{Introduction}

We introduce in this paper the notion of generating function for special
elements of a Riesz space called natural elements. Recently Ben Amor and
Omrani published a paper in which they treated the notion of moment generating
function in Riesz spaces; the main purpose of their paper is to get a
generalization of Chernoff Inequality in the setting of free-measure theory of
Riesz spaces \cite{L-888}. Our appraoch here is more general and leads to a
more general Chernoff Inequality. Also we are able to give a generalization of
a result obtained by Kuo, Vardy, and Watson \cite{L-02}. In these two papers
the authors studied only Bernoulli processes and sums of independent Bernoulli
processes which give Binomial processes. The main idea in our work here is to
develop a general approach to study discrete processes in Riesz spaces. For
this reason Functional Calculus via Daniell integral as they are developed in
\cite{L-06} and \cite{L-180} are more suitable than Functional Calculus as
they are developed in \cite{L-42}. This allows to work in the universal
completion $E^{u}$ of $E,$ and even in its sup-completion $E^{s}$. When
dealing with natural elements we can use their distributions or their
generating functions to define convergence in distribution. This notion can be
defined via order and unbounded order convergence and it will be useful to
note that these modes of convergence agree for sequences if the space is
universally complete; this result is due to Grobler and can be found in
\cite{L-65}. The latter paper contains several important results about
unbounded order convergence.

The paper is organized as follows. Section 2 contains some preliminaries
needed in the following sections. We prove some results concerning band
projections and functional calculus developed using Daniell integral on Riesz
spaces following Grobler \cite{L-06}. In Section 3 we develop some results on
series in Riesz spaces. We state and prove versions of Monotone Convergence
Theorem, Dominated Convergence Theorem and others and we will make a frequent
use of Fubini Theorem.

In section 4 we introduce the notion of generating function in Riesz spaces
equipped with a conditional expectation operator and study its basic
properties. We prove, in particular, that a natural element $x$ belongs to
$L^{p}\left(  T\right)  $ if and only if its generating function $g_{x}$ is an
order deravative of rank $k$ at $1.$

The last section contains some applications. In particular we give a new short
proof of a result established by Kuo, Vardy and Watson in \cite{L-02}
concernin Poisson approximation and we state and prove a Riesz space version
of Chernoff Inequality, generalizing a result obtained recently by by Ben Amor
and Omrani in \cite{L-888}.

For more information about generating function classical probability theory we
refer the reader to Feller \cite{b-34a}, \cite{b-1859}. All unexplained
terminology and notation concerning Riesz spaces an be found in standard
references \cite{b-240} and \cite{b-1089}.

\section{Preliminaries}

\begin{definition}
We say that $\left(  E,e,T\right)  $ is a conditional Riesz triple if $E$ is a
Dedekind complete Riesz space, $e$ is a weak order unit and $T$ is a
conditional expectation operator, that is, a positive order continuous
projection on $E$ satisfying $Te=e,$ with range, $R\left(  T\right)  ,$ a
Dedekind complete Riesz subspace. We will require in this paper that $T$ is
strictly positive, which means that if $f>0$ then $Tf>0$ (see \cite{L-180}).
\end{definition}

\textbf{Independence. }The concept of independence with respect to a
conditional expectation has been introduced in \cite[Definition 4.1 ]{L-33}.
See also \cite[Definition 4.1]{L-09}. We recall here some facts.

A family $(P_{\alpha})_{\alpha\in A}$ of band projections is said to be
$T$-independent if for each finite subset $\{\alpha_{1},...,\alpha_{n}\}$ of
$A$ we have $TP_{\alpha_{1}}P_{\alpha_{2}}....P_{\alpha_{n}}T=TP_{\alpha_{1}%
}TP_{\alpha_{2}}...TP_{\alpha_{2}}T,$ or equivalently,%
\[
TP_{\alpha_{1}}P_{\alpha_{2}}....P_{\alpha_{n}}e=TP_{\alpha_{1}}TP_{\alpha
_{2}}...TP_{\alpha_{2}}e.
\]

A family of closed Riesz subspaces $\left(  E_{\alpha}\right)  $ is said to be
$T$-independent if every family of band projections such that $P_{\alpha}e\in
E_{\alpha}$ for every $\alpha$ is $T$-independent.

A family of elements $(f_{i})$ in $E$ is independent with respect to $T$ and
$e$ if the family $\left\{  \left\langle f_{i},R\left(  T\right)
\right\rangle :i\in I\right\}  $ of Dedekind complete Riesz subspaces is
independent, where $\left\langle f,R\left(  T\right)  \right\rangle $ denotes
the closed Riesz subspace generated by $f$ and $R\left(  T\right)  .$ It is
worth noting that if $x,y\in L^{1}\left(  T\right)  $ are independent then
$xy\in L^{1}\left(  T\right)  $ and
\begin{equation}
T\left(  xy\right)  =TxTy. \label{Eq2}%
\end{equation}
Observe that the second member exists as $R\left(  T\right)  $ is an
$f$-subalgebra of $E^{u}$ (see \cite[Theorem 3.1]{L-180}). The formula is true
by definition if $x$ and $y$ are components, and hence it extends by linearity
for step elements, so for positive elements by Freudenthal Theorem and order
continuity of $T.$ Finally we use the linearity to extend it for all elements
in $L^{1}\left(  T\right)  .$ Using the fact that $T$ can be extended to the
whole space of $E^{s}$ (see \cite[Corollary 6]{L-900}) it easy to see that the
formula (\ref{Eq2}) holds for all elements in $E_{+}^{u}..$ Here we use the
fact that if $x$ and $y$ are $T$-independent then $x\wedge ne$ and $y\wedge
ne$ are $T$-independent as well. For more information about $T$-independence
the reader is referred to \textbf{\cite{L-09}.}

\textbf{Bands and projections}. As $E$ is Dedekind complete with weak order
unit, every band in $E$ is a principal projection band. If $P$ is the band
projection on the band $B$ then $P^{d}=I-P$ is the band projection on $B^{d}.$
The projection band generated by $x$ is denoted by $B_{x}$ and the
corresponding band projection is denoted by $P_{x}.$ We recall that for
$x,y\in E_{+}$ we have $B_{x\vee y}=B_{x+y}.$ It is easily seen that
$P_{x}^{d}P_{y}^{d}=P_{x\vee y}^{d}.$ If $x,y\in E$ we put $P_{x=y}=P_{\left(
x-y\right)  ^{+}}^{d}P_{\left(  y-x\right)  ^{+}}^{d}=P_{\left\vert
x-y\right\vert }^{d}.$ We shall adopt some notations used by Grobler (see for
example \cite{L-06}):

$P_{x\geq y}=P_{\left(  y-x\right)  ^{+}}^{d},$ $P_{x>y}=P_{\left(
x-y\right)  ^{+}},$

$P_{x\geq y,a\geq b}=P_{\left(  y-x\right)  ^{+}}^{d}P_{\left(  b-a\right)
^{+}}^{d}$ etc.

\textbf{Daniell integral. }For an element $x\in E$ we denote by $I_{x}$ the
Daniell integral with respect to the right continuous spectral system
$(e_{t}^{x})_{t\in%
\mathbb{R}
}:=(P_{x\leq te}e)_{t\in%
\mathbb{R}
}$ of $x$, initially defined on the Riesz space of $\mathcal{F}(%
\mathbb{R}
)$-step functions on $%
\mathbb{R}
$ denoted by $L=L(\mathcal{F}(%
\mathbb{R}
))$ and extended to $L_{2}(I_{x}),$ where $\mathcal{F}(%
\mathbb{R}
)$ is the algebra consisting of all finite unions of disjoint left open right
closed intervals $(a,b],(a,\infty),(-\infty,b]$ with $a,b\in%
\mathbb{R}
.$ We just recall that $I_{x}\left(  f\right)  $ is denoted by $f\left(
x\right)  $ and that $f\left(  x\right)  =e_{b}^{x}-e_{a}^{x}$ for
$f=\mathbf{1}_{(a,b]}.$ We refer the reader to \cite{L-06} and \cite{L-180}
for more details of the subject. We consider here the following spaces

$L^{\uparrow}=L^{\uparrow}(I_{x}):=\left\{  f\in\overline{%
\mathbb{R}
}^{%
\mathbb{R}
}:\exists\left(  f_{n}\right)  \subset L,\text{ }f_{n}\uparrow f\right\}  ,$

$L_{1}^{\uparrow}=L_{1}^{\uparrow}(I_{x}):=\left\{  f\in%
\mathbb{R}
_{+}^{%
\mathbb{R}
}:f\in L^{\uparrow}(I_{x})\text{ \ and \ }I_{x}(f)\in E\right\}  ,$

$L_{1}=L_{1}(I_{x}):=L_{1}^{\uparrow}(I_{x})-L_{1}^{\uparrow}(I_{x}),$

$L_{2}=L_{2}(I_{x}):=L^{\uparrow}(I_{x})-L_{1}(I_{x})=L^{\uparrow}(I_{x}%
)_{+}-L_{1}(I_{x}).$

\begin{lemma}
\label{YY3-A}Let $\left(  E,e,T\right)  $ be a conditional Riesz triple and
$x,y,w\in E.$ Then the following hold.

\begin{enumerate}
\item[(i)] $P_{\left\vert x-y\right\vert \vee(x-w)^{+}}=P_{\left\vert
x-y\right\vert \vee(y-w)^{+}}$;

\item[(ii)] $P_{x=y,x=w}=P_{x=y,y=w};$

\item[(iii)] $P_{x\geq y}P_{x=w}=P_{w\geq y}P_{x=w}$ and $P_{y\geq x}%
P_{x=w}=P_{y\geq w}P_{x=w}.$
\end{enumerate}
\end{lemma}

\begin{proof}
(i) It follows from the inequality%
\[
(x-w)^{+}\leq(x-y)^{+}+(y-w)^{+}\leq\left\vert x-y\right\vert +(y-w)^{+}.
\]
that\ $B_{\left\vert x-y\right\vert \vee(x-w)^{+}}\subset B_{\left\vert
x-y\right\vert +(y-w)^{+}}=B_{\left\vert x-y\right\vert \vee(y-w)^{+}}.$ We
interchange $x$ and $y,$ obtaining the converse inclusion and so the desired
equality follows.

(ii) As $P_{x=y,x=w}=P_{\left\vert x-w\right\vert \vee\left\vert
x-y\right\vert }^{d}$ it is sufficient to show the inequality
\[
P_{\left\vert x-w\right\vert \vee\left\vert x-y\right\vert }\leq P_{\left\vert
x-y\right\vert \vee\left\vert y-w\right\vert },
\]
as the converse inequality is similar. But this follows easily from the
inequality $\left\vert x-w\right\vert \leq\left\vert x-y\right\vert
+\left\vert y-w\right\vert .$ The proof of (iii) is similar.
\end{proof}

We need in our work some results on calculus functional from the paper of
Grobler \cite{L-06} and Azouzi and Trabelsi \cite{L-180}. The next result will
be useful and can not be found in those papers and is more general than
\cite[Lemma 2.7]{L-180} and \cite[Theorem 3.9]{L-06}.

\begin{theorem}
Let $E$ be a Dedekind complete Riesz space with weak order unit $e.$

\begin{enumerate}
\item[(i)] If $f,g\in L_{2}\left(  I_{x}\right)  ^{+}$ then $\left(
fg\right)  \left(  x\right)  =f\left(  x\right)  g\left(  x\right)  .$

\item[(ii)] Assume $g\in L_{1}\left(  I_{x}\right)  ,$ $f\in L_{2}\left(
I_{x}\right)  ,$ $g\left(  t\right)  \in dom\left(  f\right)  $ for every $t.$
Then $\left(  f\circ g\right)  \left(  x\right)  =f\left(  g\left(  x\right)
\right)  \in E^{s}.$
\end{enumerate}
\end{theorem}

\begin{proof}
(i) Assume first that $f=\mathbf{1}_{\left(  -\infty,a\right)  }$ and
$g=\mathbf{1}_{\left(  -\infty,b\right)  }.$ Then $fg=\mathbf{1}_{\left(
-\infty,a\wedge b\right)  }.$ Moreover $f\left(  x\right)  =p_{a},$ $g\left(
x\right)  =p_{b}$ and $\left(  fg\right)  \left(  x\right)  =p_{a\wedge
b}=p_{a}p_{b}.$ The formula is then true for all positive linear combination
of indicator functions.

If $f\in L^{\uparrow}$ there exists a sequence of step functions $\left(
f_{n}\right)  $ such that $f_{n}\uparrow f$ and $g_{n}\uparrow g.$ Now using
the previous step we deduce that%
\[
\left(  fg\right)  \left(  x\right)  =f\left(  x\right)  g\left(  x\right)  .
\]
This can be extended for every $f,g\in L_{2}\left(  I_{x}\right)  ^{+}.$

(ii) This is true if $f$ is continuous by \cite[Theorem 3.9]{L-06}. To see
that this is also true for every $f\in L^{\uparrow}$ it is only sufficient to
observe that for every $f\in L^{\uparrow}$ there is a sequence of positive
continuous functions $f_{n}$ such that $f_{n}\uparrow f.$ Now by taking the
supremum over $n$ in the equality $\left(  f_{n}\circ g\right)  \left(
x\right)  =f_{n}\left(  g\left(  x\right)  \right)  $ we obtain $\left(
f\circ g\right)  \left(  x\right)  =f\left(  g\left(  x\right)  \right)  .$
The result for general $f\in L_{2}\left(  I_{x}\right)  $ follows easily.
\end{proof}

\begin{lemma}
\label{YY3-B}Let $E$ be a Dedekind complete Riesz space with weak order unit
$e$ and $x,y\in E.$ Then the following statements hold.

\begin{enumerate}
\item[(i)] $P_{x=y}f\left(  x\right)  =P_{x=y}f\left(  y\right)  ,$ for all
$f\in L_{2}\left(  I_{x}\right)  .$ In particular $P_{x=y}x=P_{x=y}y.$

\item[(ii)] $P_{f(x)^{+}}g(x)=\left(  g\mathbf{1}_{(0,\infty)}\circ f\right)
(x)$ and $P_{f(x)^{-}}g(x)=\left(  g\mathbf{1}_{(-\infty,0)}\circ f\right)
(x),$ for all $f,g\in L_{1}\left(  I_{x}\right)  $\ for all $f,g\in
L_{1}(I_{x}).$
\end{enumerate}
\end{lemma}

\begin{proof}
(i) We prove first the equality for \ $f=\mathbf{1}_{(-\infty,a]},$ $a\in%
\mathbb{R}
.$ But in this case $f(x)=P_{\left(  x-ae\right)  ^{+}}^{d}$ and the result
follows from Lemma \ref{YY3-A}(i). So it can be extended by linearity and then
using the continuity of order projection and monotone convergence of Daniell
integral to every function in $L_{2}(I_{x}).$

(ii)\ We have%
\[
P_{f(x)^{+}}g(x)=g(x)P_{f(x)^{+}}e=g(x).\mathbf{1}_{\left(  0,\infty\right)
}\left(  f(x)\right)  =\left(  g.\mathbf{1}_{(0,\infty)}\circ f\right)  (x),
\]
The second part is similar.
\end{proof}

\begin{remark}
It should be noted that (ii) is not true for $f$ in $L_{2}\left(
I_{x}\right)  .$ Take for example, $x=e,$ $g=1,$ $f=\infty.$ Then $f\in
L_{2}\left(  I_{x}\right)  $ and $e=P_{f(x)^{+}}g(x)\neq\left(  g\mathbf{1}%
_{(0,\infty)}\circ f\right)  \left(  x\right)  =0$.
\end{remark}

\begin{definition}
Let $\left(  E,e,T\right)  $ be a conditional Riesz triple and $x\in E.$

\begin{enumerate}
\item The function%
\[
\pi_{x}:t\in%
\mathbb{R}
\longmapsto TP_{x=te}e=F_{x}(t)-F_{x}(t^{-}).
\]
is called the $(T,e)$-mass function of $x,$ where\ $F_{x}:t\longmapsto
TP_{x\leq te}e$ \ is the $T$-distribution function of $x$ and\textbf{ }%
$F_{x}(t^{-})=\lim\limits_{s\uparrow t}F_{x}\left(  s\right)  =TP_{\left(
x-te\right)  ^{-}}e.$

\item An element $x$ of $E$ is said to be natural if $%
{\textstyle\sum\limits_{n=0}^{\infty}}
\pi_{x}(n)=e.$

\item Note that if $x$ is a natural element in $E,$ then its $T$-distribution
function $F_{x}$ is piecewise-constant. Bands $\left\{  B_{x=ne}%
:n\geq0\right\}  $ are disjoint and it follows from the equality $e=%
{\textstyle\sum\limits_{n=1}^{\infty}}
\pi_{x}\left(  n\right)  $ that $E=\left[  \left\{  B_{x=ne}:n\geq0\right\}
\right]  .$ \textit{We} derive that for every $a\in\left(  n,n+1\right)  ,$
$B_{ne<x\leq ae}=0,$ which means that $F_{x}\left(  a\right)  =F_{x}\left(
n\right)  $ for every $a\in\left(  n,n+1\right)  .$

\item In the definition above an element $x$ is said to be natural if $%
{\textstyle\sum\limits_{n\in\mathbb{N}}}
\pi_{x}(n)=e.$ As $T$ is supposed to be strictly positive this condition is
equivalent to the fact that $%
{\textstyle\sum\limits_{n\in\mathbb{N}}}
P_{x=ne}e=e.$ But as it is not common to add this assumption we will keep the
first definition.
\end{enumerate}
\end{definition}

Let $\left(  E,e,T\right)  $ be a conditional Riesz triple. We say that two
natural vectors $x$ and $y$ in $E$ have the same conditional distribution with
respect to $T$ and we write \ $x\overset{d}{=}_{T}y$ \ if they\ have the same
distribution function. Bernoulli processes and Poisson processes have been
introduced in \cite{L-02}. Each one of them is defined as a sequence of
elements in $E$ having the same distribution despite the notion of
distribution was not explicitly used. Next we define in a more general manner
these notions and then explain the link between our definitions and those
introduced in \cite{L-02}.

\begin{definition}
\label{YY3-D}We say that an element $x\in E$ has Bernoulli distribution with
parameter $0<p<e$ with respect to $T,$ and write $x\sim\mathcal{B}(p).$
whenever $\pi_{x}(1)=p$ and $\pi_{x}(0)=e-p.$ Binomial distribution with
parameters $n\in%
\mathbb{N}
$ and $p\in\left(  0,e\right)  $ and Poisson distribution with parameter $g\in
E,$ $\ g>0$ are defined respectively by their mass function%
\begin{align*}
\pi_{x}\left(  k\right)   &  =\dbinom{n}{k}p^{k}\left(  e-p\right)
^{n-k},\qquad k=0,1,2,...n,\\
\pi_{x}\left(  k\right)   &  =\dfrac{g^{k}}{k!}e^{-g},\qquad k=0,1,2,....
\end{align*}
We write $x\sim\mathcal{B}(n,p)$ and $x\sim\mathcal{P}(p).$
\end{definition}

Next we show that our definition of Bernoulli processes in vector lattices
coincides with that considered in \cite{L-02}.

\begin{proposition}
Let $\left(  E,e,T\right)  $ be a conditional Riesz triple and $x\in E.$ Then
$x\sim_{T}B(p)$ if and only if $x=Qe$\ for some band projection $Q$ with
$TQe=p.$
\end{proposition}

\begin{proof}
If $x=Qe$, then\ $\pi_{x}(1)=TP_{x=e}e=TP_{Qe=e}e=TP_{e-Qe}^{d}e=TQe,$and
\ $\pi_{x}(0)=TP_{x=0}e=TP_{Qe=0}e=TP_{Qe}^{d}e=TQ^{d}e=e-TQe,$ Hence\ $x\sim
_{T}B(p)$ with $p=TQe.$ Conversely assume that $\ x\sim_{T}B(p).$ Then
$e=TP_{x=e}e+TP_{x=0}e=T\left(  P_{x=e}e+P_{x=0}\right)  e.$ We will show now
that $P_{x=e}e\wedge P_{x=0}e=0.$ But this follows from Lemma \ref{YY3-A}.(ii)
as $P_{x=0}P_{x=e}=P_{0=e}P_{x=e}=0.$ Hence $e\geq P_{x=e}e+P_{x=0}e;$\ which
gives%
\[
0=Te-TP_{x=e}e-TP_{x=0}e=T\left(  e-P_{x=e}e-P_{x=0}e\right)  .
\]
As $T$ is strictly positive and $e-P_{x=e}e-P_{x=0}e\geq0$ we get
$e=P_{x=e}e+P_{x=0}e.$ Thus
\[
x=P_{x=e}x+P_{x=0}x=P_{x=e}e,
\]
and, as required, the band projection $Q=P_{x=e}$ satisfies $TQe=TP_{x=e}e=p$.
\end{proof}

We end this preliminary section with a brief discussion concerning order
differentiability in Riesz spaces. We need this essentially to get in terms of
differentiability of the generating function in $1$ of the existence of a
moments of a given order for a natural element in $E.$

\begin{definition}
Let\ $J$ be an interval of $%
\mathbb{R}
$ not reduced to a point, $s_{0}$ a fixed point in $J$ and $E$ a Riesz space.
A function $f:J\longrightarrow E$ is said to be (order) differentiable at
$s_{0}$ if the order limit $\lim\limits_{s\longrightarrow s_{0}}%
\dfrac{f(s)-f(s_{0})}{s-s_{0}}$ \ exists in $E.$ This limit is called the
(order) derivative of $f$ at $s_{0}$ and is denoted by\ $f^{^{\prime}}(s_{0})$.
\end{definition}

Notions of right and left differentiability can be defined in an obvious way.
We can define by induction the nth derivative of $f$ by putting $f^{(n+1)}%
:=(f^{(n)})^{\prime}$ for $n=0,1,2,...$

For general information about Riesz spaces we refer the reader to
\cite{b-1089} and \cite{b-240}. We will make use of the notion of
sup-completion, a comprehensive presentation of the subject can be found in
\cite{L-444} and \cite{L-900}.

\section{Series in Riesz spaces}

We will prove in this section some results concerning series in Riesz spaces.
We consider only convergence in order, so $%
{\textstyle\sum\limits_{n=1}^{\infty}}
x_{n}$ means the order limit of the sequence $\left(  S_{n}\right)  ,$ where
$S_{n}=%
{\textstyle\sum\limits_{k=1}^{n}}
x_{k}.$ We say that the series $\sum x_{n}$ is absolutely convergent if the
series $\sum\left\vert x_{n}\right\vert $ is order convergent. In $\sigma
$-order complete Riesz spaces every absolutely convergent series is order
convergent. In particular if $\left\vert x_{n}\right\vert \leq y_{n}$ for each
$n$ and if the series $\sum y_{n}$ is order convergent, then $\sum x_{n}$ is
order convergent as well. When $\left(  f_{n}\right)  $ is a sequence of
functions from an interval $I$ to a Riesz space $E,$ we say that $\sum f_{n}$
is order convergent uniformly on $t$ if the series $\sum f_{n}\left(
t\right)  $ converges in order for every $t\in I$ and there is a net $\left(
y_{\alpha}\right)  $ such that $y_{\alpha}\downarrow0$ and for every $\alpha$
there exists $n_{0}$ such that $\left\vert \sum\limits_{k=n}^{\infty}%
f_{n}\left(  t\right)  \right\vert \leq y_{\alpha}$ for all $n\geq n_{0}.$ We
are interested in this work only in Dedekind complete Riesz spaces; in such
spaces the former property is equivalent to the fact that $\sup\limits_{n\geq
k}\sup\limits_{t\in I}\left\vert \sum\limits_{j=n}^{\infty}f_{j}\left(
t\right)  \right\vert \downarrow0$ as $k\longrightarrow\infty.$ The following
is a version of Monotone Convergence Theorem for series in Riesz spaces.

\begin{theorem}
\label{MON}Let $E$ be a Dedekind complete Riesz space and $\left(  x_{\alpha
}\left(  n\right)  \right)  _{\alpha\in\Gamma}$ be a sequence of increasing
net in $E_{+}^{s}.$ Then
\[
\lim\limits_{\alpha}%
{\textstyle\sum\limits_{n=1}^{\infty}}
x_{\alpha}\left(  n\right)  =%
{\textstyle\sum\limits_{n=1}^{\infty}}
\lim\limits_{\alpha}x_{\alpha}\left(  n\right)  =%
{\textstyle\sum\limits_{n=1}^{\infty}}
x\left(  n\right)  .
\]

\end{theorem}

\begin{proof}
Recall that suprema commute in $E_{+}^{s}$ and that the limit of an increasing
net is its supremum. As $x_{\alpha}\left(  n\right)  $ is positive for all
$n,\alpha$ we have%
\begin{align*}%
{\textstyle\sum\limits_{n=1}^{\infty}}
x\left(  n\right)   &  =\sup\limits_{N}%
{\textstyle\sum\limits_{n=1}^{N}}
x\left(  n\right)  =\sup\limits_{N}\sup\limits_{\alpha}%
{\textstyle\sum\limits_{n=1}^{N}}
x_{\alpha}\left(  n\right) \\
&  =\sup\limits_{\alpha}\sup\limits_{N}%
{\textstyle\sum\limits_{n=1}^{N}}
x_{\alpha}\left(  n\right)  =\sup\limits_{\alpha}%
{\textstyle\sum\limits_{n=1}^{\infty}}
x_{\alpha}\left(  n\right) \\
&  =\lim\limits_{\alpha}%
{\textstyle\sum\limits_{n=1}^{\infty}}
x_{\alpha}\left(  n\right)  ,
\end{align*}
and we are done.
\end{proof}

Next we state a Dominated Convergence Theorem for series in Riesz spaces.

\begin{theorem}
\label{DOM}Let $\left(  x_{\alpha}\left(  n\right)  \right)  _{\alpha\in
\Gamma}$ be a sequence of nets in $E$, where $E$ is a Dedekind complete Riesz
space. We assume that the following conditions are satisfied.

\begin{enumerate}
\item[(i)] $x_{\alpha}\left(  n\right)  \overset{o}{\longrightarrow}x\left(
n\right)  $ for each $n.$

\item[(ii)] $\left\vert x_{\alpha}\left(  n\right)  \right\vert \leq y\left(
n\right)  $ for every $\alpha,$ with $%
{\textstyle\sum\limits_{n=1}^{\infty}}
y\left(  n\right)  \in E^{u}.$\newline Then $\lim\limits_{\alpha}%
{\textstyle\sum\limits_{n=1}^{\infty}}
x_{\alpha}\left(  n\right)  =%
{\textstyle\sum\limits_{n=1}^{\infty}}
\lim\limits_{\alpha}x_{\alpha}\left(  n\right)  .$
\end{enumerate}
\end{theorem}

\begin{proof}
It is sufficient to show that%
\[
\limsup\limits_{\alpha}\left\vert
{\textstyle\sum\limits_{n=1}^{\infty}}
x_{\alpha}\left(  n\right)  -%
{\textstyle\sum\limits_{n=1}^{\infty}}
x\left(  n\right)  \right\vert =0.
\]
Observe first that $\left\vert x\left(  n\right)  \right\vert \leq y\left(
n\right)  ,$ so the series $\sum x\left(  n\right)  $ is order convergent. Fix
an integer $N.$ Then it is clear that%
\[
\lim\limits_{\alpha}%
{\textstyle\sum\limits_{n=1}^{N}}
\left\vert x_{\alpha}\left(  n\right)  -x\left(  n\right)  \right\vert =0.
\]
It is enough to prove that%
\[
A:=\limsup\limits_{\alpha}\left\vert
{\textstyle\sum\limits_{n=1}^{\infty}}
x_{\alpha}\left(  n\right)  -%
{\textstyle\sum\limits_{n=1}^{\infty}}
x\left(  n\right)  \right\vert =0.
\]
For this end observe that%
\begin{align*}
A  &  \leq\limsup\limits_{\alpha}%
{\textstyle\sum\limits_{n=1}^{\infty}}
\left\vert x_{\alpha}\left(  n\right)  -x\left(  n\right)  \right\vert \\
&  \leq\limsup\limits_{\alpha}%
{\textstyle\sum\limits_{n=N+1}^{\infty}}
\left\vert x_{\alpha}\left(  n\right)  -x\left(  n\right)  \right\vert
+\limsup\limits_{\alpha}%
{\textstyle\sum\limits_{n=1}^{N}}
\left\vert x_{\alpha}\left(  n\right)  -x\left(  n\right)  \right\vert \\
&  \leq2%
{\textstyle\sum\limits_{n=N+1}^{\infty}}
2\left\vert y\left(  n\right)  \right\vert .
\end{align*}
As this happens for every $N,$ we get $A=0$ as required.
\end{proof}

We end this section by a version of Fubini Theorem in Riesz spaces. For a
family $\left(  x_{i}\right)  _{i\in J}$ of positive elements in $E$ we write:%
\[%
{\textstyle\sum\limits_{i\in I}}
x_{i}=\sup\limits_{F\text{finite }\subset I}%
{\textstyle\sum\limits_{k\in F}}
x_{k}\in E^{s}.
\]

The following properties are very similar to the real case and their proofs
will be omitted.

\begin{theorem}
Let $E$ be a Dedekind complete Riesz space and $\left(  x_{i}\right)  _{i\in
I}$ a family in $E_{+}^{s}.$ Then the following hold.

\begin{enumerate}
\item If $I=\mathbb{N}$ then $\sum_{i\in I}x_{i}=\lim\limits_{n\longrightarrow
\infty}%
{\textstyle\sum\limits_{k=1}^{n}}
x_{k}.$

\item If $\left(  J_{\alpha}\right)  _{\alpha}is$ a partition of $I$ then $%
{\textstyle\sum\limits_{i\in I}}
x_{i}=%
{\textstyle\sum\limits_{\alpha}}
{\textstyle\sum\limits_{i\in J_{\alpha}}}
x_{i}.$ In particular we have%
\begin{equation}%
{\textstyle\sum\limits_{\left(  i,j\right)  \in\mathbb{N}^{2}}}
x_{i,j}=%
{\textstyle\sum\limits_{i\in\mathbb{N}}}
\left(
{\textstyle\sum\limits_{j\in\mathbb{N}}}
x_{i,j}\right)  =%
{\textstyle\sum\limits_{j\in\mathbb{N}}}
{\textstyle\sum\limits_{i\in\mathbb{N}}}
x_{i,j}. \label{Eq1}%
\end{equation}

\item In the last statement if the family is in $E$ then if $%
{\textstyle\sum\limits_{\left(  i,j\right)  \in\mathbb{N}^{2}}}
\left\vert x_{i,j}\right\vert \in E^{u}$ then the formula (\ref{Eq1}) holds.
\end{enumerate}
\end{theorem}

\section{Generating function}

We assume throughout that $\left(  E,e,T\right)  $ is a conditional Riesz triple.

\begin{definition}
Let $x$ be a $T$-natural vector in $E.$

\begin{enumerate}
\item[(i)] The $T$-generating function of $x$, is the function $g_{x}%
:[0,\infty)\longrightarrow E^{s}$ defined by%
\[
g_{x}\left(  s\right)  =%
{\textstyle\sum\limits_{n\geq0}}
s^{n}TP_{x=ne}e=%
{\textstyle\sum\limits_{n\geq0}}
s^{n}\pi_{x}(n).
\]

\item[(ii)] The generalized $T$-generating function of $x$ is the function
$\widetilde{g}_{x}:E_{+}\longrightarrow E^{s}$ defined by%
\[
\widetilde{g}_{x}\left(  u\right)  =%
{\textstyle\sum\limits_{n\geq0}}
TP_{x=ne}u^{n}=%
{\textstyle\sum\limits_{n\geq0}}
\pi_{x}(n)u^{n}.
\]

\end{enumerate}
\end{definition}

\begin{example}
We refer to Example \ref{YY3-D} for the notations below.

\begin{enumerate}
\item[(i)] If $x\sim B(p)$ then $g_{x}(s)=sp+e-p.$

\item[(ii)] If $x\sim B(n,p)$ then $g_{x}(s)=\left(  sp+e-p\right)  ^{n}.$

\item[(iii)] If $x\sim\mathcal{P}\left(  p\right)  $ then $g_{x}\left(
s\right)  =\exp(\left(  s-1\right)  g).$
\end{enumerate}
\end{example}

The following result collects basic properties of generating functions.

\begin{proposition}
\label{YY3-C}Let $\left(  E,e,T\right)  $ be a conditional Riesz triple and
let $x$ be a natural element in $E.$ Then the following hold.

\begin{enumerate}
\item[(i)] $g_{x}(0)=\pi_{x}(0),$ $g_{x}(1)=e$ and $g_{x}(s)=\widetilde{g}%
_{x}(se);$

\item[(ii)] $0\leq g_{x}(s)\leq e,$ \ \ $0\leq s\leq1.$ and $0\leq
\widetilde{g}_{x}(u)\leq e,$\ $0\leq u\leq e.$

\item[(iii)] $g_{x}(s)=Ts^{x},$ for all $s\geq0.$

\item[(iv)] $g_{x}$ and $\widetilde{g}_{x}$ \ are non-decreasing. Moreover, if
\ $x>0$ \ then $g_{x}$ and $\widetilde{g}_{x}$ are strictly increasing.

\item[(v)] $g_{x}$ and $\widetilde{g}_{x}$ are convex. Moreover, if $TP_{x\leq
e}e<e$ then $g_{x}$ \ and $\widetilde{g}_{x}$ are strictly convex.
\end{enumerate}
\end{proposition}

\begin{proof}
(i), (ii) and (iv) are obvious.

(iii) As $s^{ne}=s^{n}e,$ it follows from Lemma \ref{YY3-B} that
$P_{x=ne}s^{n}e=P_{x=ne}s^{x}.$ Hence%
\begin{align*}
g_{x}\left(  s\right)   &  =%
{\textstyle\sum\limits_{n=1}^{\infty}}
TP_{x=ne}s^{n}e=%
{\textstyle\sum\limits_{n=1}^{\infty}}
TP_{x=ne}s^{x}\\
&  =T\left[  \left(
{\textstyle\sum\limits_{n=1}^{\infty}}
P_{x=ne}\right)  s^{x}\right]  =Ts^{x}.
\end{align*}

(v) This follows from the fact that the map $s\longmapsto s^{n}$ is convex for
$n\geq0$ and strictly convex for $n\geq2$. Observe that the condition
$TP_{x\leq e}e<e$ means that $\pi_{x}\left(  n\right)  >0$ for some $n\geq2.$
\end{proof}

Chernoff Inequality is stated in general in terms of exponential generating
functions (see for instance \cite{b-3075}). We state a version of Chernoff
Inequality in Riesz spaces.

\begin{theorem}
\label{YY3-K}(Chernoff Bounds) Let $\left(  E,e,T\right)  $ be a Riesz
conditional triple, $x$ a natural element in\ $E$ and $u\in R\left(  T\right)
.$ Then the following hold.%
\[
TP_{x\geq u}e\leq s^{u}g_{x}(s),\ \ \forall s>1,\
\]%
\[
TP_{x\leq u}e\leq s^{-u}g_{x}(s),\ \ \forall0<s<1.
\]

\end{theorem}

\begin{proof}
Let $s>1.$ According to Lemma \ref{YY3-A}.(iii) we have%
\[
P_{x\geq u}P_{x=ne}s^{u}=P_{ne\geq u}P_{x=ne}s^{u}e\leq P_{ne\geq u}%
P_{x=ne}s^{n}e.
\]
It follows that%
\begin{align*}
P_{x\geq u}s^{u}e  &  =\sum\limits_{n=1}^{\infty}P_{x\geq u}P_{x=ne}s^{u}%
e\leq\sum\limits_{n=0}^{\infty}P_{ne\geq u}P_{x=ne}s^{n}e\\
&  =P_{ne\geq u}\sum\limits_{n=0}^{\infty}P_{x=ne}s^{n}e=P_{ne\geq u}x^{s}.
\end{align*}
\newline Applying $T$ to both sides, we get the first inequality. The second
one can be proved in a similar way.
\end{proof}

\begin{corollary}
\label{YY3-L}Let $x$ be a natural vector in\ $E$ and $\alpha\in\mathbb{R}%
_{+}.$ Then the following hold.%
\[
TP_{x\geq\alpha e}e\leq\dfrac{g_{x}(s)}{s^{\alpha}},\ \ \forall s>1\
\]%
\[
TP_{x\leq\alpha e}e\leq\dfrac{g_{x}(s)}{s^{\alpha}},\ \ \forall0<s<1.
\]

\end{corollary}

We will show now that the main result in \cite{L-888} can be deduced from the
previous corollary.

\begin{theorem}
\cite[Theorem 5.3]{L-888}Let $\left(  E,e,T\right)  $ be Riesz conditional
triple. Let $\left(  P_{j}\right)  $ be a Benoulli process with $TP_{j}=f$ for
all $j\in\mathbb{N}$ for some $f\in E.$ Let $S_{n}=%
{\textstyle\sum\limits_{k=1}^{n}}
P_{k}.$ Then%
\[
TP_{\left(  S_{n}-te\right)  ^{+}}e\leq\left(  \dfrac{n\left\Vert f\right\Vert
_{u}\exp1}{t}\right)  ^{t}\exp\left(  -nf\right)  .
\]
holds for all $t>n\left\Vert f\right\Vert _{e},$ where%
\[
\left\Vert f\right\Vert _{e}:=\inf\left\{  \beta\in%
\mathbb{R}
:\left\vert f\right\vert \leq\beta e\right\}  .
\]

\end{theorem}

It should be noted that in \cite{L-888} the authors assume that $e$ is a
strong unit for $E.$ This is the case if we work in the ideal generated by
$e.$

\begin{proof}
As $S_{n}\sim_{T}B(n,f)$ we have%
\[
g_{S_{n}}(s)=\left(  sf+e-f\right)  ^{n}=\left(  e+\left(  s-1\right)
f\right)  ^{n}.
\]
In the sequel we will use the inequality $\log\left(  e+x\right)  \leq x$ for
all $x\in E_{+}$ and the formula (see \cite[Theorem 3.9]{L-06})%
\[
\exp\left[  n\log x\right]  =x^{n}%
\]
true for any $x\geq\alpha e,$ where $\alpha\in\left(  0,\infty\right)  .$ Note
also that $e$ is a multiplicative unit in the $f$-algebra $E_{e}.$ The real
$s=\dfrac{t}{n\left\Vert f\right\Vert _{e}}$ satisfies $s>1$ and we have then%
\begin{align*}
g_{S_{n}}(s)  &  =\left(  e+\left(  s-1\right)  f\right)  ^{n}=\exp\left[
n\log\left(  e+\left(  s-1\right)  f\right)  \right] \\
&  \leq\exp\left[  n\left(  s-1\right)  f\right]  =\exp\left[  nsf\right]
\exp\left[  -nf\right] \\
&  \leq\exp\left(  ns\left\Vert f\right\Vert _{e}e\right)  \exp\left(
-nf\right)  =\exp\left(  t\right)  \exp\left(  -nf\right)  .
\end{align*}
Now using Corollary \ref{YY3-L} to derive that%
\begin{align*}
TP_{\left(  S_{n}-te\right)  ^{+}}e  &  \leq TP_{S_{n}\geq te}e\leq
\dfrac{g_{S_{n}}(s)}{s^{t}}\leq\dfrac{\exp t.\exp\left(  -nf\right)  }{s^{t}%
}\\
&  =\exp t.\left(  \frac{n\left\Vert f\right\Vert _{e}}{t}\right)  ^{t}%
\exp\left(  -nf\right)  ,
\end{align*}
which completes the proof.
\end{proof}

\begin{lemma}
Let $\left(  E,e,T\right)  $ be a conditional Riesz triple and $x$ a natural
element in $E.$ Then the following hold.

\begin{enumerate}
\item[(i)] $g_{x}$ is continuous on $\left[  0,1\right]  $ and infinitely
differentiable on $[0,1).$ Moreover we have the formula%
\begin{equation}
g_{x}^{(k)}(s)=%
{\textstyle\sum\limits_{n=k}^{\infty}}
n(n-1)...(n-k+1)s^{n-k}\pi_{x}(n) \label{F1}%
\end{equation}
for every integer $k\geq0$ and for all $s\in\lbrack0,1).$

\item[(ii)] For each integer $k$, we have%
\[
g_{x}^{(k)}(s)=T\left[  x(x-e)...(x-(k-1)e)s^{x-ke}\right]  .
\]
In particular, $g_{x}^{(k)}(0)=k!\pi_{x}(k).$
\end{enumerate}
\end{lemma}

It follows then from the above lemma that

\begin{center}%
\[
g_{x}(s)=%
{\textstyle\sum\limits_{n=0}^{\infty}}
s^{n}\dfrac{g_{x}^{(k)}(0)}{k!},\qquad s\in\lbrack0,1).
\]

\end{center}

\begin{proof}
(i) As $\left\vert s^{k}\pi_{x}\left(  k\right)  \right\vert \leq\pi
_{x}\left(  k\right)  $ and $\sum\pi_{x}\left(  k\right)  $ order converges,
it follows from Theorem \ref{DOM} that%
\[
\lim\limits_{s\longrightarrow s_{0}}\sum\limits_{n=0}^{\infty}s^{n}\pi
_{x}\left(  k\right)  =\sum\limits_{n=0}^{\infty}\lim\limits_{s\longrightarrow
s_{0}}s^{n}\pi_{x}\left(  k\right)  =\sum\limits_{n=0}^{\infty}s_{0}^{n}%
\pi_{x}\left(  k\right)  ,
\]
which proves that $g_{x}$ is continuous on $\left[  0,1\right]  .$ Let us now
fix a real $s_{0}\in\lbrack0,1)$ and a compact neighborhood $K$ of $s_{0}$ in
$[0,1).$ We have for $k\in\mathbb{N}$ and $s\in K,$%
\[
\left\vert \dfrac{s^{k}\pi_{x}\left(  k\right)  -s_{0}^{k}\pi_{x}\left(
k\right)  }{s-s_{0}}\right\vert \leq kM^{k-1}\pi_{x}\left(  k\right)  \leq
kM^{k-1}e,
\]
where $M=\sup\limits_{t\in K}\left\vert t\right\vert <1.$ Using Theorem
\ref{DOM} again we derive that $g_{x}$ is differentiable on $[0,1)$ and that
formula \ref{F1} holds for $k=1$. We can then prove by induction on $k$ that
$g_{x}$ is $k$-differentiable and the formula \ref{F1} is valid.

(ii) Let \ $0\leq s<1.$\ We have%
\begin{align*}
g_{x}^{(k)}(s)  &  =%
{\textstyle\sum\limits_{n=k}^{\infty}}
n(n-1)...(n-k+1)s^{n-k}\pi_{x}(n)\\
&  =Th_{k}(x)=T\left[  x(x-e)...(x-(k-1)e)s^{x-ke}\right]  ,
\end{align*}
where $h_{k}(u)=u(u-1)...(u-k+1)\mathbf{1}_{\left[  k,\infty\right)  }(u).$
\end{proof}

The next lemma generalizes a well-known fact in Probability theory and will be
needed in the proof of Proposition \ref{YY3-F} below.

\begin{lemma}
If $x$ is a natural element in $E$ then
\[
Tx=%
{\textstyle\sum\limits_{n=1}^{\infty}}
TP(x\geq ne)e.
\]

\end{lemma}

\begin{proof}
We have%
\[
Tx=%
{\textstyle\sum\limits_{n=1}^{\infty}}
n\pi_{x}(n)=%
{\textstyle\sum\limits_{n=1}^{\infty}}
{\textstyle\sum\limits_{k=1}^{n}}
\pi_{x}(n)=%
{\textstyle\sum\limits_{k=1}^{\infty}}
{\textstyle\sum\limits_{n=k}^{\infty}}
\pi_{x}(n)=%
{\textstyle\sum\limits_{k=1}^{\infty}}
TP(x\geq ke)e,
\]
which proves the lemma.
\end{proof}

\begin{proposition}
\label{YY3-F}Let $\left(  E,e,T\right)  $ be a conditional Riesz triple and
let $x$ be a natural element in\ $E.$ For each $n\in%
\mathbb{N}
$ \ the following are equivalent:

\begin{enumerate}
\item[(i)] $Tx^{n}\in E^{u}$ (i.e. $x\in L^{n}(T)$);

\item[(ii)] $g_{x}^{(n)}(1)$ exists in $E^{u}$ (in the sense that the left
limit $\lim_{s\uparrow1}g_{x}^{(n)}(s)$ exists in $E^{u}$).\newline In either
case, we have%
\begin{align*}
g_{x}^{(n)}(1)  &  =T\left[  x(x-e)...(x-(n-1)e)\right] \\
&  =%
{\textstyle\sum\limits_{k=n}^{\infty}}
k(k-1)...(k-n+1)\pi_{x}(k).
\end{align*}

\end{enumerate}
\end{proposition}

The quantities \ $Tx^{n},$ $n\geq1,$ are called \ $T$-moments of $x,$ and the
quantities $Tx,Tx(x-e),Tx(x-e)(x-2e),...$are called $T$-factorials moments of
$\ x.$

\begin{proof}
Observe that
\begin{align*}
\dfrac{g_{x}(s)-g_{x}(1)}{s-1}  &  =\dfrac{e-g_{x}(s)}{1-s}=%
{\textstyle\sum\limits_{n=1}^{\infty}}
\frac{1-s^{n}}{1-s}\pi_{x}(n)\\
&  =%
{\textstyle\sum\limits_{n=1}^{\infty}}
\left(  1+s+...+s^{n-1}\right)  \pi_{x}(n).
\end{align*}
So according to Monotone Convergence Theorem (Theorem \ref{MON}) we have%
\begin{align*}
\lim\limits_{s\uparrow1}\dfrac{g_{x}(s)-g_{x}(1)}{s-1}  &  =%
{\textstyle\sum\limits_{n=1}^{\infty}}
\lim\limits_{s\uparrow1}\left(  1+s+...+s^{n-1}\right)  \pi_{x}(n)\\
&  =%
{\textstyle\sum\limits_{n=1}^{\infty}}
n\pi_{x}(n)=Tx.
\end{align*}
This proves the equivalence for $n=1.$

Assume now that the result is true for rank $n.$ Let us prove it for rank
$n+1.$ Assume first that $g^{\left(  n+1\right)  }\left(  1\right)  $ exists.
Then%
\begin{align*}
g^{\left(  n+1\right)  }\left(  1\right)   &  =\lim\limits_{s\uparrow1}%
\dfrac{g_{x}^{\left(  n\right)  }(s)-g_{x}^{\left(  n\right)  }(1)}{s-1}%
=\lim\limits_{s\uparrow1}%
{\textstyle\sum\limits_{k=n}^{\infty}}
k\left(  k-1\right)  ...\left(  k-n+1\right)  \dfrac{\left(  1-s^{n-k}\right)
}{1-s}\pi_{x}(k)\\
&  =\left(  \ast\right)
{\textstyle\sum\limits_{k=n}^{\infty}}
\lim\limits_{s\uparrow1}k\left(  k-1\right)  ...\left(  k-n+1\right)  \left(
k-n\right)  \pi_{x}(k)=T\left[  x\left(  x-e\right)  ...\left(  x-ne\right)
\right]  ,
\end{align*}
where the equality $\left(  \ast\right)  $ follows from Theorem \ref{MON}. Now
it follows from Lyapunov inequality that if $x\in L^{q}\left(  T\right)  $
then $x\in L^{p}\left(  T\right)  $ for all $p\leq q.$ So as $x\in
L^{n}\left(  T\right)  $ we have $Tx^{n+1}\in E^{u}$ if and only if $T\left[
x\left(  x-e\right)  ...\left(  x-ne\right)  \right]  \in E^{u}.$This proves
the result.

Assume now that $Tx^{n+1}\in E^{u}$ then by the above remark, $Tx^{n}\in
E^{u}$ and then by the induction hypothesis, $g^{\left(  n\right)  }\left(
1\right)  $ exists. Also $T\left[  x\left(  x-e\right)  ...\left(
x-ne\right)  \right]  \in E^{u},$ and then the use of Theorem \ref{MON} gives
again that%
\[
g^{\left(  n+1\right)  }\left(  1\right)  =T\left[  x\left(  x-e\right)
...\left(  x-ne\right)  \right]  \in E^{u}.
\]
The proof is now complete.
\end{proof}

\begin{corollary}
Let $\left(  E,e,T\right)  $ be a conditional Riesz triple and $x$ a natural
element in\ $E.$ Then $Tx=g_{x}^{\prime}(1)$ and $Tx^{2}=g_{x}^{\prime
}(1)+g_{x}^{\prime\prime}(1)$ if $Tx^{2}\in E^{u}.$
\end{corollary}

\section{Generating function and T-independence}

\begin{theorem}
Let $\left(  x_{k}\right)  _{1\leq k\leq n}$ be a sequence of $T$-independent
natural elements in\ $E,$ and put $S_{n}:=%
{\textstyle\sum\limits_{i=1}^{n}}
x_{i},$ then the $T$-generating function of $S_{n}$ is given by
\[
g_{S_{n}}(s)=Ts^{S_{n}}=g_{x_{1}}(s)...g_{x_{n}}(s).
\]

\end{theorem}

\begin{proof}
It is enough to prove the result for the case $n=2.$ Let $x$ and $y$ be two
natural elements in $E$ that are $T$-independent. We have
\begin{align*}
g_{x+y}(s)  &  =%
{\textstyle\sum\limits_{n\geq0}}
s^{n}TP_{x+y=ne}e=%
{\textstyle\sum\limits_{n\geq0}}
s^{n}%
{\textstyle\sum_{k=0}^{n}}
TP_{y=ke}P_{x+y=ne}e\\
&  =%
{\textstyle\sum\limits_{n\geq0}}
s^{n}%
{\textstyle\sum\limits_{k=0}^{n}}
TP_{y=ke}P_{x=\left(  n-k\right)  e}e=%
{\textstyle\sum\limits_{n\geq0}}
s^{n}%
{\textstyle\sum\limits_{i+j=n}}
TP_{y=je}e.TP_{x=ie}e\\
&  =\left(
{\textstyle\sum\limits_{i\geq0}}
s^{i}TP_{x=ie}e\right)  .\left(
{\textstyle\sum\limits_{j\geq0}}
s^{j}TP_{y=je}e\right)  =g_{x}(s).g_{y}(s).
\end{align*}
\newline This proves that $g_{x+y}\left(  s\right)  =g_{x}\left(  s\right)
g_{y}\left(  s\right)  $ as required.
\end{proof}

Let $N,x_{0},x_{1},x_{2},...$ be natural elements in $E$ with $x_{0}=0.$ Put
$x_{N}:=%
{\textstyle\sum\limits_{n\geq1}}
P_{N=ne}x_{n},$ $S_{n}:=%
{\textstyle\sum\limits_{k=0}^{n}}
x_{k}$ and $S_{N}:=%
{\textstyle\sum\limits_{n\geq1}}
P_{N=ne}S_{n}.$ It follows from these definitions that $P_{N=ne}x_{N}%
=P_{N=ne}x_{n}$ and $P_{N=ne}S_{N}=P_{N=ne}S_{n}.$

\begin{lemma}
\label{YY3-G}For any $f\in L_{2}(x)$ the following statements hold.

\begin{enumerate}
\item[(i)] We have for every $n,k\in\mathbb{N}$%
\[
P_{N=ne}f(x_{N})=P_{N=ne}f(x_{n}).
\]
In particular we have%
\[
P_{N=ne}P_{x_{N}=ke}e=P_{N=ne}P_{x_{n}=ke}e.
\]

\item[(ii)] We have for every $k,n$ in $\mathbb{N},$%
\[
P_{N=ne}f(S_{N})=P_{N=ne}f(S_{n})\ \text{and}\ P_{N=ne}P_{S_{N}=ke}%
e=P_{N=ne}P_{S_{n}=ke}e.
\]

\end{enumerate}
\end{lemma}

\begin{proof}
(i) It suffices to show these equalities for $f=\mathbf{1}_{\left(
-\infty,a\right]  },$ $a\in%
\mathbb{R}
.$ By definition of $x_{N}$ it follows that $P_{N=ne}x_{N}=P_{N=ne}x_{n}.$
Using this and noting that $f(x)=P_{x\leq ae}e,$ we get%
\begin{align*}
P_{N=ne}f(x_{N})  &  =P_{N=ne}P_{\left(  x_{N}-ae\right)  ^{+}}^{d}%
e=P_{N=ne}\left(  I-P_{\left(  x_{N}-ae\right)  ^{+}}\right)  e\\
&  =P_{N=ne}-P_{P_{N=ne}\left(  x_{N}-ae\right)  ^{+}}e.
\end{align*}
\newline Now as%
\begin{align*}
P_{N=ne}\left(  x_{N}-ae\right)  ^{+}  &  =\left(  P_{N=ne}x_{N}%
-P_{N=ne}ae\right)  ^{+}\\
&  =\left(  P_{N=ne}x_{n}-P_{N=ne}ae\right)  ^{+}=P_{N=ne}\left(
x_{n}-ae\right)  ^{+}%
\end{align*}
we deduce that%
\[
P_{N=ne}f(x_{N})=P_{N=ne}-P_{P_{N=ne}\left(  x_{n}-ae\right)  ^{+}}%
e=P_{N=ne}f(x_{n}).
\]
as required. To get the second equality we just apply the first one
for\ $f=\mathbf{1}_{\left\{  k\right\}  }.$

(ii) The proof is similar to (i).
\end{proof}

\begin{lemma}
\label{YY3-I}With the above notations we have:

\begin{enumerate}
\item[(i)] $x_{N}$ and $S_{N}$ are natural elements in $E.$

\item[(ii)] $S_{N}=%
{\textstyle\sum\limits_{k\geq1}}
P_{N\geq ke}x_{k}.$
\end{enumerate}
\end{lemma}

\begin{proof}
(i) We will show the requested property only for $x_{N};$ the case of $S_{N}$
is similar. By Lemma \textbf{\ref{YY3-G}, }we get%
\begin{align*}%
{\textstyle\sum\limits_{k\geq0}}
P_{x_{N}=ke}e  &  =%
{\textstyle\sum\limits_{k\geq0}}
{\textstyle\sum\limits_{n\geq0}}
P_{N=ne}P_{x_{N}=ke}e=%
{\textstyle\sum\limits_{k\geq0}}
{\textstyle\sum\limits_{n\geq0}}
P_{N=ne}P_{x_{n}=ke}e\\
&  =%
{\textstyle\sum\limits_{n\geq0}}
P_{N=ne}%
{\textstyle\sum_{k\geq0}}
P_{x_{n}=ke}e=%
{\textstyle\sum\limits_{n\geq0}}
P_{N=ne}e=e.
\end{align*}
\newline This show that $x_{N}$ is natural.

(ii) By definition of $S_{N}$ we have%
\begin{align*}
S_{N}  &  =%
{\textstyle\sum\limits_{n\geq0}}
P_{N=ne}S_{n}=%
{\textstyle\sum\limits_{n\geq0}}
P_{N=ne}%
{\textstyle\sum\limits_{k=1}^{n}}
x_{k}\\
&  =%
{\textstyle\sum\limits_{k\geq1}}
\left(
{\textstyle\sum_{n\geq k}}
P_{N=ne}\right)  x_{k}=%
{\textstyle\sum\limits_{k\geq1}}
P_{N\geq ke}x_{k}.
\end{align*}
\newline This gives the desired formula.
\end{proof}

\begin{theorem}
Assume that $x=x_{1},x_{2},...$ are $T$-equidistributed and $T$-conditional
independent of $N.$ The generating function of $x_{N}$ is given by:%
\[
g_{x_{N}}=T(I-P_{N=0})g_{x}+TP_{N=0}e.
\]
Moreover, if $TP_{N>0}e=e$ then $g_{x_{N}}=g_{x}.$
\end{theorem}

\begin{proof}
We have%
\begin{align*}
g_{x_{N}}(s)  &  =%
{\textstyle\sum\limits_{n\geq0}}
s^{n}TP_{x_{N}=ne}e=%
{\textstyle\sum\limits_{n\geq0}}
s^{n}%
{\textstyle\sum\limits_{k\geq0}}
TP_{N=ke}P_{x_{k}=ne}e\\
&  =%
{\textstyle\sum\limits_{n\geq0}}
s^{n}TP_{N=0}P_{x_{0}=ne}e+%
{\textstyle\sum\limits_{n\geq0}}
s^{n}%
{\textstyle\sum\limits_{k\geq1}}
TP_{N=ke}.TP_{x_{k}=ne}e\\
&  =TP_{N=0}e+%
{\textstyle\sum\limits_{k\geq1}}
TP_{N=ke}.%
{\textstyle\sum\limits_{n\geq0}}
s^{n}TP_{x=ne}e.
\end{align*}
\newline Now if we assume that\ $TP_{N>0}e=e,$ we get\ $TP_{N=0}e=0,$ and then
$g_{x_{N}}(s)=Tg_{x}(s)=g_{x}(s).$
\end{proof}

\begin{theorem}
\label{YY3-J}Let $\left(  E,e,T\right)  $ be a conditional Riesz triple.
Assume that $x_{1},x_{2},...$ are $T$-equidistributed and $T$-conditional
independent of\ $N.$ Then

\begin{enumerate}
\item[(i)] the generating function of $S_{N}$ is given by the following
formula:%
\[
g_{S_{N}}=\widetilde{g}_{N}\circ g_{x}.
\]

\item[(ii)] If $N$ and $x$ belong to $L^{1}\left(  T\right)  ,$ then
$TS_{N}=\left(  TN\right)  .\left(  Tx\right)  .$

\item[(iii)] If $N$ and $x$ belong to $L^{2}\left(  T\right)  ,$ then%
\[
\operatorname*{Var}\nolimits_{T}\left(  S_{N}\right)  =\left(  TN\right)
.\operatorname*{Var}\nolimits_{T}\left(  x\right)  +\operatorname*{Var}%
\nolimits_{T}\left(  N\right)  .\left(  Tx\right)  ^{2},
\]
where,
\[
\operatorname*{Var}\nolimits_{T}\left(  y\right)  :=T\left(  y-Ty\right)
^{2}=Ty^{2}-\left(  Ty\right)  ^{2}.
\]

\end{enumerate}
\end{theorem}

\begin{proof}
(i) Using Lemma \ref{YY3-G}.(ii), we get%
\begin{align*}
g_{S_{N}}(s)  &  =%
{\textstyle\sum\limits_{n\geq0}}
TP_{N=ne}s^{S_{n}}=%
{\textstyle\sum\limits_{n\geq0}}
T\left(  s^{S_{n}}.P_{N=ne}e\right) \\
&  =%
{\textstyle\sum\limits_{n\geq0}}
Ts^{S_{n}}.TP_{N=ne}e=%
{\textstyle\sum\limits_{n\geq0}}
\left(  Ts^{x}\right)  ^{n}.TP_{N=ne}e=\widetilde{g}_{N}\left(  g_{x}%
(s)\right)  .
\end{align*}

(ii) As $S_{n}$ and $N$ are $T$-independent we have for each $n,$
$TP_{N=ne}S_{n}=TS_{n}.TP_{N=ne}e.$ Hence%
\begin{align*}
TS_{N}  &  =%
{\textstyle\sum\limits_{n\geq1}}
TP_{N=ne}S_{n}=%
{\textstyle\sum\limits_{n\geq1}}
TS_{n}.TP_{N=ne}e=%
{\textstyle\sum\limits_{n\geq1}}
nTx.TP_{N=ne}e\\
&  =Tx.%
{\textstyle\sum\limits_{n\geq1}}
nTP_{N=ne}e=Tx.TN.
\end{align*}

(iii) We have on one hand,%
\[
\left(  TS_{N}\right)  ^{2}=\left(  Tx\right)  ^{2}.\left(  TN\right)  ^{2}.
\]
and on the other hand we have%
\begin{align*}
T\left(  S_{N}^{2}\right)   &  =%
{\textstyle\sum\limits_{n\geq1}}
Tx_{n}^{2}.TP_{N\geq ne}e+%
{\textstyle\sum\limits_{n\neq m\geq1}}
\left(  Tx_{n}x_{m}\right)  .TP_{N\geq ne}P_{N\geq me}e\\
&  =Tx^{2}.%
{\textstyle\sum\limits_{n\geq1}}
TP_{N\geq ne}e+\left(  Tx\right)  ^{2}.%
{\textstyle\sum\limits_{n\neq m\geq1}}
TP_{N\geq ne}P_{N\geq me}e\\
&  =\left(  Tx^{2}\right)  .\left(  TN+Z\right)  .
\end{align*}
\newline It is enough then to show that $Z:=%
{\textstyle\sum\limits_{n\neq m\geq1}}
TP_{N\geq ne}P_{N\geq me}e=TN^{2}-TN.$ To this end observe that%
\begin{align*}
Z  &  =%
{\textstyle\sum\limits_{k=n\vee m,n\neq m}}
TP_{N\geq ke}e=%
{\textstyle\sum\limits_{m=2}^{\infty}}
2(m-1)TP_{N\geq ke}e\\
&  =2%
{\textstyle\sum\limits_{k>1}}
\left(
{\textstyle\sum\limits_{m=1}^{k-1}}
m\right)  P_{N=ke}e=%
{\textstyle\sum\limits_{k=1}}
\left(  k^{2}-k\right)  P_{N=ke}e=TN^{2}-TN.
\end{align*}
\newline This completes the proof.
\end{proof}

\section{Application to convergence}

It is easily seen that if $x$ and $y$ are natural elements in $E$ having the
same generating function then $x$ and $y$ have the same $T$-distribution,
namely $TP_{x=ke}e=TP_{y=ke}e$ for $k=0,1,...$ Several notions have been
generalized from the classical theory of probability and stochastic processes
to the setting of Riesz spaces. One of the challenges is to give a reasonable
definition of convergence in distribution. Many equivalent definitions are
known in the literature. In \cite{b-23}, David Williams defines the
convergence in distribution via distribution function: A sequence $\left(
X_{n}\right)  $ of real random variables with distribution functions $\left(
F_{X_{n}}\right)  $ is said to converge in distribution to $X$ if $F_{X_{n}%
}\left(  t\right)  \longrightarrow F\left(  t\right)  $ for every
$t\in\operatorname*{Cont}\left(  F\right)  ,$ where $F_{X}$ is the
distribution function of $X$ and $\operatorname*{Cont}\left(  F\right)  $ is
the set of points on which $F$ is continuous. According to \cite[Example
2.4]{L-329}, there is a Riesz space $E$ and an element $x$ in $E$ such that
$F_{x}$ is discontinuous everywhere. So the definition above can not be
translated to the theory of Riesz spaces. We suggest, however, the following
definition in the case of natural elements.

\begin{definition}
Let $\left(  E,e,T\right)  $ be a conditional Riesz triple and $\left(
x_{\alpha}\right)  $ a net of neutral elements in $E.$ We saythat the net
$\left(  x_{\alpha}\right)  _{\alpha\in A}$ converges in $T$-distribution to a
natural element $x$ in $E$\ if $TP_{x_{\alpha}=ke}e\overset{o}{\longrightarrow
}TP_{x=ke}e$ for all $k\geq0.$ We write $x_{\alpha}\overset{T-d}%
{\longrightarrow}x.$
\end{definition}

\begin{proposition}
Let $\left(  E,e,T\right)  $ be a conditional Riesz triple, $\left(
x_{\alpha}\right)  $ a net of natural elements in $E$ and $x\in E.$ Then the
following statements are equivalent.

\begin{enumerate}
\item[(i)] The net $\left(  x_{\alpha}\right)  $ converges to $x$ in $T$-distribution.

\item[(ii)] The net $\left(  g_{x_{\alpha}}\right)  $ converges to $g_{x}$
pointwise on $\left[  0,1\right]  $.
\end{enumerate}
\end{proposition}

\begin{proof}
Let us denote $TP_{x_{\alpha}=ke}e$ briefly by $p_{k,\alpha}$, and
$TP_{x=ke}e$ by $p_{k}.$

(i) $\Longrightarrow$ (ii) The conclusion is trivial for $s=1.$ Let $s$
be\ fixed in $[0,1)$. We have for each $k,$ $\lim\limits_{\alpha}$
$s^{k}p_{k,\alpha}=s^{k}p_{k}$ and $\left\vert s^{k}p_{k,\alpha}\right\vert
\leq s^{k}e.$ Now as the series $\sum\limits_{k\geq0}s^{k}e$ is order
convergent in $E_{e}$ we can apply Theorem \ref{DOM} to get\ $g_{x_{\alpha}%
}(s)\overset{o}{\longrightarrow}g_{x}(s).$

(ii) $\Longrightarrow$ (i) The proof will be done by induction on
$k$.\ Observe first that\ $p_{0}^{(\alpha)}=g_{x_{\alpha}}\left(  0\right)
\overset{o}{\longrightarrow}g_{x}\left(  0\right)  =p_{0}$. Assume by
induction that $p_{i}^{(\alpha)}\overset{o}{\longrightarrow}p_{i}$ for all
$i=0,1,...k$ and let $s\in\left(  0,1\right)  .$ Define
\[
A_{n}^{\alpha}(s)=\left\vert p_{k+1}^{(\alpha)}-p_{k+1}+\sum\limits_{i\geq
k+2}\left(  p_{i}^{(\alpha)}-p_{i}\right)  s^{i-(k+1)}\right\vert .
\]
\newline Then%
\begin{align*}
A_{n}^{\alpha}\left(  s\right)   &  =\left\vert \frac{g_{x_{\alpha}}%
(s)-\sum_{i=0}^{k}p_{i}^{(\alpha)}s^{i}}{s^{k+1}}-\frac{g_{x}(s)-\sum
_{i=0}^{k}p_{i}s^{i}}{s^{k+1}}\right\vert \\
&  \leq\left\vert \frac{g_{x_{\alpha}}(s)-g_{x}(s)}{s^{k+1}}\right\vert
+\left\vert \frac{\sum_{i=0}^{k}p_{i}^{(\alpha)}s^{i}-\sum_{i=0}^{k}p_{i}%
s^{i}}{s^{k+1}}\right\vert \overset{o}{\longrightarrow}0.
\end{align*}
It follows that%
\[
\limsup_{\alpha}\left\vert p_{k+1}^{(\alpha)}-p_{k+1}\right\vert \leq\left(
\frac{s}{1-s}\right)  e+\limsup_{\alpha}A_{n}^{\alpha}(s),
\]
\ \ for all $s\in\left(  0,1\right)  .$ Hence $p_{k+1}^{(\alpha)}\overset
{o}{\longrightarrow}p_{k+1}$ as required and the proof is complete.
\end{proof}

\begin{theorem}
Suppose that $N\backsim_{T}\mathcal{P}(g)$, and that $x_{1},x_{2},...$ are
$T$-independent, identically distributed with Bernoulli distribution
$\mathcal{B}\left(  p\right)  $ and independent of $N$. Let $S_{n}%
=\sum\limits_{i=1}^{n}x_{i}$. Then $S_{N}$ has the Poisson distribution with
parameter $pg$.
\end{theorem}

\begin{proof}
By Theorem \ref{YY3-J},%
\begin{align*}
g_{S_{N}}(s)  &  =\widetilde{g}_{N}\left(  g_{x_{1}}(s)\right)  =\exp(\left(
g_{x_{1}}(s)-e\right)  g)\\
&  =\exp(\left(  e-p+sp-e\right)  g)=\exp(\left(  -p+sp\right)  g)\\
&  =\exp(-pg+spg)=\exp(-\left(  1-s\right)  pg).
\end{align*}
This shows that $S_{N}$ has the Poisson distribution with parameter $pg$ as desired.
\end{proof}

As the function $\exp$ is convex on $\mathbb{R}$ it follows from
\cite[Corollary 4.3]{L-06} that $\exp\left(  x\right)  $ is well defined in
$E^{s}$ for all $x\in E.$ Observe that if $x$ is positive then $\exp\left(
x\right)  =\sup\limits_{n}v_{n},$ where $v_{n}=\sum\limits_{k=0}^{n}%
\dfrac{x^{k}}{k!}.$ We will show that $\exp x$ is in fact in $E^{u},$ the
universal completion of $E.$ To this end it is sufficient to prove that
$\left(  v_{n}\right)  $ is unbounded order Cauchy (see \cite[Theorem
3.10]{L-65} or \cite[Theorem 17]{L-444}). Now observe that for $p\leq q,$ and
$u\in E^{+},$%
\begin{align*}
\left\vert v_{p}-v_{q}\right\vert \wedge u  &  =P_{\left(  x-te\right)  ^{+}%
}\left[  \left(  v_{q}-v_{p}\right)  \wedge u\right]  +P_{\left(  x-te\right)
^{+}}^{d}\left[  \left(  v_{q}-v_{p}\right)  \wedge u\right] \\
&  \leq P_{\left(  x-te\right)  ^{+}}u+%
{\textstyle\sum\limits_{k=p+1}^{q}}
\left(  t^{k}/k!\right)  e.
\end{align*}
Hence%
\[
\limsup\limits_{p,q}\left[  \left\vert v_{p}-v_{q}\right\vert \wedge u\right]
\leq P_{\left(  x-te\right)  ^{+}}u.
\]
As this holds for every $t\geq0$ we deduce that $\left(  v_{n}\right)  $ is
unbounded order Cauchy as required. It should be noted that if $x\in E_{e},$
the ideal generated by $e,$ then one can work in that ideal which is lattice
isomorphic to some $C\left(  K\right)  $-space (see for instance
\cite[Theorems 6.2 and 6.8]{b-2538}) and then use different tool which may
seem easier. In both cases the formula $\exp\left(  x\right)  \exp\left(
-x\right)  =e$ is valid. In the sequel we extend a useful result verified by
the exponentiation to the setting of Riesz spaces.

\begin{lemma}
\label{YY3-E}Let $E$ be a Dedekind complete Riesz space with weak order unit
$e$ and $\left(  x_{n}\right)  $ a sequence in $E$ such that $x_{n}^{n}\in E$
for each $n.$ Then the following hold.

\begin{enumerate}
\item[(i)] If $x_{n}\overset{o}{\longrightarrow}0$ then $x_{n}^{n}\overset
{uo}{\longrightarrow}0$.

\item[(ii)] If\ $x_{n}\overset{o}{\longrightarrow}x$ \ in $E_{e}$ then
$\left(  1+\dfrac{x_{n}}{n}\right)  ^{n\text{ \ }}\overset{o}{\longrightarrow
}\exp(x)$ in $E^{u}.$
\end{enumerate}
\end{lemma}

\begin{proof}
(i)\ As $e$ is a weak order unit it is sufficient to prove that $x_{n}%
^{n}\wedge e\overset{o}{\longrightarrow}0$ \cite[ Corollary 3.5]{L-65}. Let us
fix a real $t$ in $\left(  0,1\right)  .$ As $\left(  x_{n}\right)  $ is order
null, $P_{\left(  x_{n}-te\right)  ^{+}}e$ is order null as well. We derive
that%
\begin{align*}
x_{n}^{n}\wedge e  &  =\left(  x_{n}\wedge e\right)  ^{n}=\left(  P_{\left(
x_{n}-e\right)  ^{+}}\left(  x_{n}\wedge e\right)  +P_{\left(  x_{n}%
-te\right)  ^{+}}^{d}\left(  x_{n}\wedge e\right)  \right)  ^{n}\\
&  =P_{\left(  x_{n}-te\right)  ^{+}}\left(  x_{n}\wedge e\right)
^{n}+P_{\left(  x_{n}-te\right)  ^{+}}^{d}\left(  x_{n}\wedge e\right)  ^{n}\\
&  \leq P_{\left(  x_{n}-te\right)  ^{+}}e+P_{\left(  x_{n}-te\right)  ^{+}%
}^{d}x_{n}^{n}\leq P_{\left(  x_{n}-te\right)  ^{+}}e+t^{n}e.
\end{align*}
This shows that $x_{n}^{n}\wedge e\overset{o}{\longrightarrow}0$ and then (i)
is proved.

(ii) Write $\left(  1+\dfrac{x_{n}}{n}\right)  ^{n}=\sum\limits_{k=0}^{\infty
}f_{n}\left(  k\right)  $ where
\[
f_{n}\left(  k\right)  =\left\{
\begin{array}
[c]{cc}%
\dbinom{n}{k}\dfrac{x_{n}^{k}}{n^{k}} & \text{if }0\leq k\leq n\\
0 & \text{otherwise}%
\end{array}
\right.  .
\]
Let $M$ be an upper bound of $\left(  \left\vert x_{n}\right\vert \right)  .$
Then it is quite easy to see that $\left\vert f_{n}\left(  k\right)
\right\vert \leq\dfrac{M^{k}}{k!}$ and $f_{n}\left(  k\right)  \overset
{o}{\longrightarrow}\dfrac{x^{k}}{k!}$ as $n\longrightarrow\infty.$ It follows
from Theorem \ref{DOM} that $\left(  1+\dfrac{x_{n}}{n}\right)  ^{n}%
\overset{o}{\longrightarrow}e^{x}$ in $E^{u}$ as required.
\end{proof}

\begin{remark}
In (i) the order convergence can not be expected; it is only true in $E^{u}$
as order convergence and unbounded order convergence agree in $E^{u}$ for
sequences by a result of Grobler (see \cite[Corollary 3.12]{L-65}). The order
convergence of $\left(  x_{n}\right)  $ can not guarantee the order
boundedness of the sequence $\left(  x_{n}^{n}\right)  .$ Take for instance
$E=\ell^{\infty}$ and $x_{n}=\left(  1+1/\sqrt{n}\right)  e_{n},$ where
$\left(  e_{n}\right)  $ is the standard basis. Then $x_{n}\overset
{o}{\longrightarrow}0,$ however $x_{n}^{n}\geq\sqrt{n}e_{n},$ hence $\left(
x_{n}^{n}\right)  $ is not order bounded in $\ell^{\infty}.$
\end{remark}

The following theorem generalizes a well-known result in probability theory
(see for instance, \cite[Example 9.5.2]{b-1859}). A Riesz space version of
this result has been proved by Kuo, Vardy, and Watson \cite[Theorem 5.6]%
{L-02}, which is a particular case of our result.

\begin{theorem}
Let $\left(  E,e,T\right)  $ be a conditional Riesz triple and $\left(
y_{n}\right)  $ a sequence of natural elements in $E.$ Assume that $y_{n}$ has
a Binomial $T$-distribution $\mathcal{B}(n,p_{n})$ such that $np_{n}%
\overset{o}{\longrightarrow}g$ for some $g$ with $0<g\in E_{e}^{+}.$ Then
$\left(  y_{n}\right)  $ converges in $T$-distribution to $\mathcal{P}\left(
g\right)  $, that is,
\[
TP_{y_{n}=ke}e\overset{o}{\longrightarrow}\dfrac{1}{k!}g^{n}\exp\left(
-g\right)
\]
for all integer $k.$
\end{theorem}

\begin{proof}
We have for all $s\in\left[  0,1\right]  ,$%
\[
g_{y_{n}}(s)=\left(  e-\left(  1-s\right)  p_{n}\right)  ^{n}=\left(
e-\dfrac{\left(  1-s\right)  np_{n}}{n}\right)  ^{n},
\]
and by applying Lemma \ref{YY3-E} to $x_{n}=\left(  1-s\right)  np_{n}$
and\ $x=\left(  1-s\right)  g$ we get%
\[
g_{y_{n}}(s)\overset{o}{\longrightarrow}\exp(\left(  s-1\right)
g),\ \text{for all }0\leq s\leq1,
\]
and the desired result follows.
\end{proof}

\end{document}